\documentclass[11pt,reqno]{amsart}
\usepackage[utf8]{inputenc}
\usepackage{amsmath,amssymb,url}
\usepackage{supertabular}
\usepackage{mathrsfs}
\usepackage{epsfig}

\newcommand{\CC}{{\mathbf C}}
\newcommand{\RR}{{\mathbf R}}
\newcommand{\QQ}{{\mathbf{Q}}}

 \DeclareMathOperator{\Prob}{\mathbf{P}}   
 \DeclareMathOperator{\E}{\mathbf{E}}      

\newcommand{\convlaw}{\overset{\mbox{\rm \scriptsize law}}{\Rightarrow}}

\def\stacksum#1#2{{\stackrel{{\scriptstyle #1}}
{{\scriptstyle #2}}}}

\newcommand{\demi}{{\textstyle{\frac{1}{2}}}}

\newcommand{\eps}{\varepsilon}
\newcommand{\Cc}{\CC}

\newcommand{\Rr}{\RR}

\newcommand{\Qq}{\QQ}
\newcommand{\Fp}{\mathbf{F}}

\newcommand{\proba}{\Prob}
\newcommand{\expect}{\E}

\newcommand{\charfun}{\mathbf{1\!\!\!1}}

\newcommand{\harm}[1]{h_{{#1}}}
\newcommand{\irred}{\pi}
\newcommand{\nirred}{|\irred|}
\newcommand{\degp}[1]{d^+({{#1}})}
\newcommand{\degm}[1]{d^-({{#1}})}
\newcommand{\ellp}[1]{\ell^+({{#1}})}
\newcommand{\ellm}[1]{\ell^-({{#1}})}
\newcommand{\erreurm}[1]{\Bigl(1+O\Bigl({{#1}}\Bigr)\Bigr)}

\newcommand{\ra}{\rightarrow}
\newcommand{\lra}{\longrightarrow}

\newcommand{\fleche}[1]{\stackrel{#1}{\lra}}

\newcommand{\sym}{\mathfrak{S}}

\def\sumb{\mathop{\sum \Bigl.^{\flat}}\limits}

\DeclareMathOperator{\spec}{Spec}
\DeclareMathOperator{\ord}{ord}

\renewcommand{\leq}{\leqslant}
\renewcommand{\geq}{\geqslant}

\newtheorem{thm}{Theorem}[section]

\newtheorem{lem}[thm]{Lemma}
\newtheorem{prop}[thm]{Proposition}

\theoremstyle{definition}
\newtheorem{defn}[thm]{Definition}
\theoremstyle{remark}
\newtheorem{rem}[thm]{Remark}
\newtheorem{example}[thm]{Example}

\numberwithin{equation}{section}

\begin{document}
\title[Mod-Poisson convergence]{Mod-Poisson convergence in probability
  and number theory}

 \author{E. Kowalski} \address{ETH Z\"urich -- D-MATH \\ R\"{a}mistrasse
   101\\ 8092 Z\"{u}rich, Switzerland}
 \email{kowalski@math.ethz.ch}

\author{A. Nikeghbali}
 \address{Institut f\"ur Mathematik,
 Universit\"at Z\"urich, Winterthurerstrasse 190,
 CH-8057 Z\"urich,
 Switzerland}
 \email{ashkan.nikeghbali@math.uzh.ch}



\begin{abstract}
  Building on earlier work introducing the notion of ``mod-Gaussian''
  convergence of sequences of random variables, which arises naturally
  in Random Matrix Theory and number theory, we discuss the analogue
  notion of ``mod-Poisson'' convergence.  We show in particular how it
  occurs naturally in analytic number theory in the classical
  Erd\H{o}s-Kac Theorem.  In fact, this case reveals deep
  connections and analogies with conjectures concerning the
  distribution of $L$-functions on the critical line, which belong to
  the mod-Gaussian framework, and with analogues over finite fields,
  where it can be seen as a zero-dimensional version of the
  Katz-Sarnak philosophy in the ``large conductor'' limit.
\end{abstract}

\maketitle

\section{Introduction} \label{section:Intro} 

In our earlier paper~\cite{jkn} with J. Jacod,\footnote{\ Although
  this new paper is largely self-contained, it is likely to be most
  useful for readers who have at least looked at the introduction and
  the examples in~\cite{jkn}, especially Section 4} motivated by
results from Random Matrix Theory and probability, we have introduced
the notion of mod-Gaussian convergence of a sequence of random
variables $(Z_N)$. This occurs when the sequence does not (typically)
converge in distribution, so the sequence of characteristic functions
does not converge pointwise to a limit characteristic function, but
nevertheless, the characteristic functions decay precisely like a
suitable Gaussian, i.e., the limits
\begin{equation}\label{eq-mod-gaussian}
\lim_{N\ra +\infty}{\exp(-iu\beta_N+u^2\gamma_N/2)\expect(e^{iuZ_N})} 
\end{equation}
exist, locally uniformly for $u\in\Rr$, for some parameters
$(\beta_N,\gamma_N)\in \Rr\times [0,+\infty[$. 
\par
Besides giving natural and fairly general instances of such behavior
in probability theory, we investigated arithmetic instances of it.  In
that respect, we noticed that the limits~(\ref{eq-mod-gaussian}) can
not exist if the random variables $Z_N$ are integer-valued, since the
characteristic functions $\expect(e^{iuZ_N})$ are then
$2\pi$-periodic, and we discussed briefly the possibility of
introducing ``mod-Poisson convergence'', that may be applicable to
such situations. Indeed, we noticed that this can be seen to occur in
number theory in one approach to the famous Erd\H{o}s-Kac Theorem.
\par
In the present paper, we look more deeply at mod-Poisson
convergence. We first recall the definition and give basic facts about
mod-Poisson convergence in Sections~\ref{sec-mod-poisson}
and~\ref{sec-2}. Sections~\ref{sec-nt} and~\ref{sec-finite-fields}
consider number-theoretic situations related to the Erd\H{o}s-Kac
Theorem. We show that the nature of the mod-Poisson convergence
parallels closely the structure of conjectures for the moments of zeta
functions on the critical line. This becomes especially clear over
finite fields, leading to very precise analogies with the Katz-Sarnak
philosophy and conjectures. In fact, in Section~\ref{sec-main}, we
prove a version of the mod-Poisson convergence for the number of
irreducible factors of a polynomial in $\Fp_q[X]$, as the degree
increases, which is a zero-dimensional case of the large conductor
limit for $L$-functions (see Remark~\ref{rm-scheme} and
Theorem~\ref{th-katz-sarnak}). Our proof convincingly explains the
probabilistic features of the limiting function, involving both local
models of primes and large random permutations.
\par
\smallskip
\par
\textbf{Notation.} In number-theoretic contexts, $p$ always refers to
a prime number, and sums and products over $p$ (with extra conditions)
are over primes satisfying those conditions.
\par
For any integer $d\geq 1$, we denote by $\sym_d$ the symmetric
group on $d$ letters and by $\sym_d^{\sharp}$ the set of its conjugacy
classes. Recall these can be identified with partitions of $d$, where
the partition
$$
n=1\cdot r_1+\cdots+ d\cdot r_d,\quad\quad r_i\geq 0,
$$
corresponds to permutations with $r_1$ fixed points, $r_2$ disjoint
$2$-cycles, ...., $r_d$ disjoint $d$-cycles.  For $\sigma\in\sym_d$,
we write $\sigma^{\sharp}$ for its conjugacy class. We denote by
$\varpi(\sigma)$ the number of disjoint cycles occurring in $\sigma$.
\par
By $f\ll g$ for $x\in X$, or $f=O(g)$ for $x\in X$, where $X$ is an
arbitrary set on which $f$ is defined, we mean synonymously that there
exists a constant $C\geq 0$ such that $|f(x)|\leq Cg(x)$ for all $x\in
X$. The ``implied constant'' refers to any value of $C$ for which this
holds. It may depend on the set $X$, which is usually specified
explicitly, or clearly determined by the context.
\par
\smallskip
\par
\textbf{Acknowledgments}. We thank P-O. Dehaye and A.D. Barbour for
interesting discussions related to this paper, and P. Bourgade for
pointing out a computational mistake in an earlier draft. Thanks also
to the referee for a careful reading of the manuscript.
\par
The second author was partially supported by SNF Schweizerischer
Nationalfonds Projekte Nr. 200021 119970/1.

\section{General properties of mod-Poisson
  convergence}\label{sec-mod-poisson}

Recall that a Poisson random variable $P_{\lambda}$ with parameter
$\lambda>0$ is one taking (almost surely) integer values $k\geq 0$
with
$$
\proba(P_{\lambda}=k)=\frac{\lambda^k}{k!} e^{-\lambda}.
$$
\par
Its characteristic function is then given by
$$
\E(e^{iuP_{\lambda}})=\exp(\lambda(e^{iu}-1)).
$$

\begin{defn}
  We say that a sequence of random variables $(Z_N)$ converges in the
  mod-Poisson sense with parameters $\lambda_N$ if the
  following limits
$$
\lim_{N\ra+\infty}{\expect(e^{iuP_{\lambda_N}})^{-1}
\expect(e^{iuZ_N})}=
\lim_{N\ra +\infty}
\exp(\lambda_N(1-e^{iu}))\E(e^{iuZ_N})=\Phi(u)
$$
exist for every $u\in\RR$, and the convergence is locally uniform. The
\emph{limiting function} $\Phi$ is then continuous and $\Phi(0)=1$.
\end{defn}

\begin{example}
  (1) The simplest case of mod-Poisson convergence (which justifies
  partly the name) is given by
\begin{equation}\label{eq-regular}
Z_N=P_{\lambda_N}+Z
\end{equation}
where $P_{\lambda_N}$ is a Poisson variable with parameter
$\lambda_N$, while $Z$ is an arbitrary random variable independent of
all $P_{\lambda_N}$. In that case, the limiting function is the
characteristic function $\expect(e^{iuZ})$ of $Z$.
\par
(2) Often, and in particular in the cases of interest in the
arithmetic part of this paper, $Z_N$ is (almost surely)
integer-valued; in that case, its characteristic function is
$2\pi$-periodic, and it follows that if the convergence is locally
uniform, then it is in fact uniform for $u\in\Rr$. However, this is
not always the case, as shown by examples like~(\ref{eq-regular})
if the fixed random variable $Z$ is not itself integer-valued.
\par
(3) A.D. Barbour pointed out to us the paper~\cite{hwang} of
H-K. Hwang. Hwang introduces an analytic assumption~\cite[(1),
p. 451]{hwang} on the \emph{probability generating functions} of
integer-valued random variables $(X_N)$, i.e., on the power series
$$
\sum_{n\geq 1}{\proba(X_N=n)z^n}=\expect(z^{X_N}),
$$
which is very closely related to mod-Poisson convergence.  This
assumption is used as a basis to deduce results on Poisson
approximation of the sequence (see Proposition~\ref{th-ks-distance}
below for a simple example).  Hwang also gives many additional
examples where his assumption holds.
\end{example}

If we have mod-Poisson convergence with parameters $(\lambda_N)$ which
converge, then $(Z_N)$ converges in law. Such a situation arises for
instance in the so-called Poisson convergence (see,
e.g.,\cite[p. 188]{breiman}), which we recall: 

\begin{prop}
Let $(X_k^{(n)})$ be an array of independent random variables,
identically distributed in each row, according to a Bernoulli
distribution with parameter $x_n$: 
$$
\Prob(X_i^{(n)}=1)=x_n \text{ and }
\Prob(X_i^{(n)}=0)=1-x_n\quad\text{ for } 1\leq i\leq n.
$$
\par
Set $S_n=X_1^{(n)}+\ldots+X_n^{(n)}$. Then, $S_n$ converges in
distribution if and only if $nx_n\to\lambda>0$, when $n\to\infty$. The
limit random variable $S$ is a Poisson random variable with parameter
$\lambda$.
\end{prop}

We will state an analogue of Poisson convergence in the mod-Poisson
setting in the next section, but first we discuss some basic
consequences.  The link with mod-Gaussian convergence in the last part
of the next result is quite intriguing.

\begin{prop}\label{prop-cor-mod-poisson}
  Let $(Z_N)$ be a sequence of random variables which converges in the
  mod-Poisson sense, with parameters $\lambda_N$, such that
$$
\lim_{N\to\infty}\lambda_N=\infty.
$$
\par
Then the following hold:
\par
\emph{(1)} The re-scaled variables $Z_N/\lambda_N$ converge in
probability to $1$, that is, for any $\varepsilon>0$,
$$
\lim_{N\to\infty}\Prob\Bigl(
\Bigl|\frac{Z_N}{\lambda_N}-1\Bigr|>\varepsilon\Bigr)=0.
$$
\par
\emph{(2)} We have the normal convergence
\begin{equation}\label{eq-clt}
\dfrac{Z_N-\lambda_N}{\sqrt{\lambda_N}}\convlaw\mathcal{N}(0,1),
\end{equation}
where $\mathcal{N}$ is a standard Gaussian random variable.
\par
\emph{(3)} The random variables
$$
Y_N=\frac{Z_N-\lambda_N}{\lambda_N^{1/3}}
$$
converge in the mod-Gaussian sense~\emph{(\ref{eq-mod-gaussian})} with
parameters $(0,\lambda_N^{1/3})$ and universal limiting function
$$
\Phi_u(t)=\exp(-it^3/6).
$$
\end{prop}

\begin{proof}
  This is a very standard probabilistic argument, but we give details
  for completeness.
\par
(1) For $u\in\RR$, we write
$$
s=\frac{u}{\lambda_N}
$$
(note that $s$ depends on $N$ and $s\ra 0$ when $N\ra +\infty$).  By
the definition of mod-Poisson convergence (in particular the uniform
convergence with respect to $u$), we have
$$
\lim_{N\ra +\infty}{\exp(\lambda_N(1-e^{is}))\E(e^{is Z_N})}=\Phi(0)=1.
$$
\par
The fact that 
$$ 
\exp(\lambda_N(e^{is}-1))=\exp((is+O(s^2))\lambda_N),
$$
yields
$$
\lim_{N\ra +\infty}{\E(e^{i u Z_N/\lambda_N})}=e^{i u}.
$$
Consequently, $(Z_N/\lambda_N)$ converges in distribution to $1$ and
hence converges in probability since the limiting random variable is
constant.
\par
(2) For $u\in\RR$, we now write
$$
t=\frac{u}{\sqrt{\lambda_N}}
$$
(note that $t$ depends on $N$ and $t\ra 0$ when $N\ra +\infty$).
\par
Again, by the definition of mod-Poisson convergence (in particular the
uniform convergence with respect to $u$), we have
\begin{equation}\label{eq-1}
\lim_{N\ra +\infty}{\exp(\lambda_N(1-e^{it}))\E(e^{it Z_N})}=\Phi(0)=1.
\end{equation}
\par
Moreover, we have
\begin{align}
  \exp(\lambda_N(e^{it}-1))
&=\exp((it-t^2/2+O(t^3))\lambda_N)\nonumber\\
&=\exp\Bigl(iu\sqrt{\lambda_N}-\frac{u^2}{2}+
O\Bigl(
\frac{u^3}{\sqrt{\lambda_N}}
\Bigr)\Bigr).
\label{eq-2}
\end{align}
\par
 Let
$$
Y_N=\frac{Z_N-\lambda_N}{\sqrt{\lambda_N}}.
$$
We have then
\begin{equation}\label{eq-charfn}
  \E(e^{iuY_N})=\exp(-iu\sqrt{\lambda_N})\E(e^{it Z_N}).
 \end{equation}
\par
Writing~(\ref{eq-charfn}) as
$$
\exp(-iu\sqrt{\lambda_N}) \times
\exp((e^{it}-1)\lambda_N)\times 
\exp((1-e^{it})\lambda_N)
\E(e^{it Z_N}) ,
$$
we see from~(\ref{eq-1}) and~(\ref{eq-2}) that this is
$$
\exp\Bigl(-\frac{u^2}{2}+O\Bigl(\frac{u^3}{\sqrt{\lambda_N}}\Bigr)
\Bigr)
(1+o(1))\ra \exp\Bigl(-\frac{u^2}{2}\Bigr),\quad \text{ as }\quad
N\ra +\infty,
$$
and by L\'evy's criterion, this concludes the proof.
\par
Part (3) is a similar straightforward computation, which we leave as
an enlightening exercise.
\end{proof}

In stating the renormalized convergence to a Gaussian variable, there
is a loss of information, since the ``Poisson nature'' of the sequence
is lost. This is illustrated further by the following result which
goes some way towards clarifying the probabilistic nature of
mod-Poisson convergence. We recall that the Kolmogorov-Smirnov
distance between real-valued random variables $X$ and $Y$ is defined
by
$$
d_{KS}(X,Y)=\sup_{x\in\Rr}{|\proba(X\leq x)-\proba(Y\leq x)|}.
$$


\begin{prop}\label{th-ks-distance}
  Let $(Z_N)$ be a sequence of random variables which are
  a.s. supported on positive integers, and which converges in the
  mod-Poisson sense for some parameters $(\lambda_N)$, such that
  $\lambda_N\to\infty$ when $N\to\infty$. Assume further that the
  characteristic functions $\expect(e^{iuZ_N})$ are of $C^1$ class and
  the convergence holds in $C^1$ topology.
\par
Then we have
$$
\lim_{N\ra +\infty}{d_{KS}(Z_N,P_{\lambda_N})}=0,
$$
where $P_{\lambda_N}$ is a Poisson random variable with parameter
$\lambda_N$, and in fact
$$
d_{KS}(Z_N,P_{\lambda_N})\leq \|\Phi'\|_{\infty}\lambda_N^{-1/2},
$$
for $N\geq 1$. 
\end{prop}

\begin{proof}
  We recall the following well-known inequality, which is the ad-hoc
  tool (see, e.g.~\cite[p. 186, 5.10.2]{Petrov95}): if $X$ and $Y$ are
  integer-valued random variables, then
$$
d_{KS}(X,Y)\leq\dfrac{1}{4}
\int_{-\pi}^\pi\left|\dfrac{\expect(e^{iuX})-\expect(e^{iuY})}{u}\right| du.
$$
\par
Let
$$
\psi_N(u)=\expect(e^{iu P_{\lambda_N}}),\quad
\Phi_N(u)=\psi_N(u)^{-1}\expect(e^{iuZ_N}).
$$
\par
From the inequality, we obtain
\begin{align*}
  d_{KS}(Z_N,P_{\lambda_N})&\leq \dfrac{1}{4}\int_{-\pi}^\pi
  {|\expect(e^{iuZ_N})-\psi_N(u)|\frac{du}{u}}\\
  &= \frac{1}{4}
  \int_{-\pi}^{\pi}{\Bigl|\psi_N(u)\frac{\Phi_N(u)-1}{u}\Bigr|du}
\end{align*}
\par
From our stronger assumption of mod-Poisson convergence with $C^1$
convergence, we have a uniform bound
$$
\Bigl|\frac{\Phi_N(u)-1}{u}\Bigr|\leq \|\Phi'_N\|_{\infty},
$$
for $N\geq 1$, hence since $|\Psi_N(u)|=\exp(\lambda_N(\cos u-1))$, we
have
$$
d_{KS}(Z_N,P_{\lambda_N})\leq \frac{\|\Phi'\|_{\infty}}{4}
\int_{-\pi}^{\pi}{e^{\lambda_N(\cos
 u-1)}du}. 
$$
\par
It is well-known that the precise asymptotic of such an integral gives
order of magnitude $\lambda_N^{-1/2}$ for $\lambda_N\ra +\infty$. To
see this quickly, note for instance that $\cos u-1\leq -u^2/5$ on
$[-\pi,\pi]$, hence
$$
\int_{0}^{\pi}{e^{\lambda_N(\cos u-1)}du}\leq
\int_{\pi}^{\pi}{e^{-\lambda_N u^2/5}du}\leq \int_{\Rr}{e^{-\lambda_N
    u^2/5}du}=\sqrt{\frac{5\pi}{\lambda_N}},
$$
which gives the result since $\sqrt{5\pi}/4\leq 1$.
\end{proof}

\begin{rem}
  (1) Hwang~\cite[Th. 1]{hwang} gives this and many other variants for
  other measures of approximation, under the assumption of his version
  of mod-Poisson convergence.  In another work with A. Barbour, we
  consider various refinements and applications of this type of
  statement, including with approximation involving more general
  families of discrete random variables (see~\cite{bkn}).
\par
(2) As a reference for number theorists, note that the existence of
renormalized convergence as in~(\ref{eq-clt}) for an arbitrary
sequence of integer-valued random variables $(Z_N)$, with
$\expect(Z_N)=\lambda_N$, does not imply that the Kolmogorov distance
$d_{KS}(Z_N,P_{\lambda_N})$ converge to $0$: indeed, consider
$$
Z_N=B_1+\cdots+B_N
$$
where the $B_i$ are Bernoulli random variables with
$\proba(B_i=1)=\proba(B_i=0)=\demi$. Then $\lambda_N=\tfrac{N}{2}$,
and the normalized convergence in law~(\ref{eq-clt}) is the Central
Limit Theorem. However, it is known that, for some constant $c>0$, we
have
$$
d_{KS}(Z_N,P_{\lambda_N})\geq c>0
$$
for all $N$ (see, e.g.,~\cite[Th. 2]{barbour-hall}, for the analogue
in total variation distance, which in that case is comparable to the
Kolmogorov distance~\cite[Prop. 1]{roos}).
\end{rem}

\section{Limit theorems with mod-Poisson behavior}
\label{sec-2}

Now we give an analogue of the Poisson convergence in the mod-Poisson
framework.

\begin{prop}\label{prop-poisson-convergence}
  Let $(x_n)$ of positive real numbers with
\begin{equation}\label{eq-mod-poisson-cond}
\sum_{n\geq 1}{x_n}=+\infty,\quad\quad \sum_{n\geq 1}{x_n^2}<+\infty, 
\end{equation}
and let $(B_n)$ be a sequence of independent Bernoulli random variables
with
$$
\proba(B_n=0)=1-x_n,\quad\quad \proba(B_n=1)=x_n.
$$
Then 
$$
Z_N=B_1+\cdots+B_N
$$
has mod-Poisson convergence with parameters
$$
\lambda_N=x_1+\cdots +x_N
$$
and with limiting function given by
$$
\Phi(u)=\prod_{n\geq 1}{(1+x_n(e^{iu}-1))\exp(x_n(1-e^{iu}))},
$$
a uniformly convergent infinite product.
\end{prop}

\begin{proof}
  This is again a quite simple computation.
Indeed, by independence of the variables $B_n$, we have
$$
\exp(\lambda_N(1-e^{iu}))\E(e^{iuZ_N})
=\prod_{n=1}^{N}\exp(x_n(1-e^{iu}))(1+x_n(e^{iu}-1)),
$$
and since
$$
\exp(x_n(1-e^{iu}))(1+x_n(e^{iu}-1))=1+O(x_n^2)
$$
for $u\in\Rr$ and $n\geq 1$ (recall $x_n\ra 0$), it follows
from~(\ref{eq-mod-poisson-cond}) that this product converges locally
uniformly to $\Phi(u)$, which completes the proof.
\end{proof}

\begin{rem}
  More generally, assume that $(X_k^{(n)})$ is a triangular array of
  independent random variables taking values in $\{0,a_1,\ldots,
  a_r\}$, such
  that $$\Prob[X_k^{(n)}=a_i]=x_n^{(i)};\;\;i=1,\ldots,r.$$ Assume
  that for any $i$, $\sum_{n\geq1}x_n^{(i)}=\infty$ and
  $\sum_{n\geq1}(x_n^{(i)})^2<\infty$. Then
  $S_n=X_1^{(n)}+\ldots+X_n^{(n)}$ converges in the mod-Poisson sense
  with parameter $\lambda_N=a_1x_n^{(1)}+\ldots+a_rx_n^{(r)}$.
\end{rem}

\section{Mod-Poisson convergence and the Erd\H{o}s-Kac
  Theorem: a first analogy}\label{sec-nt}

In~\cite[\S 4.3]{jkn}, we gave the first example of mod-Poisson
convergence as explaining (through the Central Limit of
Proposition~\ref{prop-cor-mod-poisson}) the classical result of
Erd\H{o}s and Kac concerning the statistic behavior of the
arithmetic function function $\omega(n)$, the number of (distinct)
prime divisors of a positive integer $n\geq 1$:
\begin{equation}\label{eq-erdos-kac}
\lim_{N\ra +\infty}{
\frac{1}{N}
|\{
n\leq N\,\mid\, 
a<\frac{\omega(n)-\log\log N}{\sqrt{\log\log N}}<b
\}|}
=
\frac{1}{\sqrt{2\pi}}\int_{a}^b{
e^{-t^2/2}dt}
\end{equation}
for any real numbers $a<b$. 
\par
More precisely, with
$$
\omega'(n)=\omega(n)-1,\quad \text{ for } n\geq 2,
$$ 
we showed by a simple application of the Delange-Selberg method (see,
e.g.,~\cite[II.5, Theorem 3]{tenenbaum}) that for any $u\in \RR$, we
have
$$
\lim_{N\ra +\infty}{ \frac{ (\log N)^{(1-e^{iu})}}{N}\sum_{2\leq n\leq
    N}{e^{iu\omega'(n)}}}=\Phi(u),
$$
and the convergence is uniform, with
\begin{equation}\label{eq-phi-omega}
  \Phi(u)=\frac{1}{\Gamma(e^{iu}+1)}\prod_p{
  \Bigl(1-\frac{1}{p}\Bigr)^{e^{iu}}
  \Bigl(1+\frac{e^{iu}}{p-1}\Bigr)
},
\end{equation}
where the Euler product is absolutely and uniformly convergent: this
means mod-Poisson convergence with parameters $\lambda_N=\log\log
N$. By Proposition~\ref{prop-cor-mod-poisson}, (2), this
implies~(\ref{eq-erdos-kac}).\footnote{\ As we observed, this gives
  essentially the proof of the Erd\H{o}s--Kac theorem due to R\'enyi
  and Tur\'an~\cite{renyi-turan}. For another recent simple proof,
  see~\cite{granville-sound}.} To illustrate what extra information is
contained in mod-Poisson convergence we make two remarks: first, by
putting $u=\pi$, for instance, we get
$$
\sum_{1\leq n\leq N}{(-1)^{\omega(n)}}=o\Bigl(\frac{N}{(\log N)^2}\Bigr),
$$
as $N\ra +\infty$ (since $1/\Gamma(1+e^{i\pi})=0$), which is a
statement well-known to be equivalent to the Prime Number
Theorem. Secondly, more generally, we can apply results like
Proposition~\ref{th-ks-distance} (which is easily checked to be
applicable here) to derive Poisson-approximation results for
$\omega(n)$ which are much more precise than the renormalized Gaussian
behavior (see also~\cite[\S 4]{hwang} and~\cite[\S 6.1]{tenenbaum} for
the discussion of the classical work of Sath\'e and Selberg).
\par
\medskip
\par
We wish here to bring to light the very interesting, and very
complete, analogy between the probabilistic structure of this
mod-Poisson version of the Erd\H{o}s-Kac Theorem and the
mod-Gaussian conjecture for the distribution of the values
$L$-functions, taking as basic example the conjecture for the
distribution of $\log |\zeta(1/2+it)|$, which follows from the
Keating-Snaith moment conjectures for the Riemann zeta function
(see~\cite{jkn},~\cite{keating-snaith}).
\par
We start with the observation, following from~(\ref{eq-phi-omega}),
that the limiting function $\Phi(u)$ in the Erd\H{o}s-Kac Theorem
takes the form of a product $\Phi(u)=\Phi_1(u)\Phi_2(u)$ with
$$
\Phi_1(u)=\frac{1}{\Gamma(e^{iu}+1)},\quad\quad \Phi_2(u)= \prod_{p}{
  \Bigl(1-\frac{1}{p}\Bigr)^{e^{iu}}
  \Bigl(1+\frac{e^{iu}}{p-1}\Bigr) }.
$$
\par
We compare this with the Moment Conjecture in the mod-Gaussian form,
namely, if $U$ is uniformly distributed on $[0,T]$, it is expected
that
\begin{equation}\label{eq-moments-conj}
\lim_{T\ra +\infty}{e^{u^2 \log\log T}\expect(e^{iu\log
    |\zeta(1/2+iU)|^2})}=
\Psi_1(u)\Psi_2(u),
\end{equation}
for all $u\in\Rr$ (locally uniformly) where
\begin{align}
  \Psi_1(u)&= \frac{G(1+iu)^2}{G(1+2iu)},\\
  \intertext{($G(z)$ is the Barnes double-gamma function, see
    e.g.~\cite[Ch. XII, Misc. Ex. 48]{ww}), and}
 \Psi_2(u)&=\prod_{p}{
    \Bigl(1-\frac{1}{p}\Bigr)^{-u^2}\Bigl\{ \sum_{m\geq
      0}{\Bigl(\frac{\Gamma(m+iu)}{m!\Gamma(\lambda)}\Bigr)^2
      p^{-m}}\Bigr\}}.
\end{align}
\par
Here also, the limiting function splits as a product of two terms, and
each appears individually as limit in a distinct mod-Gaussian
convergence. Indeed, we first have
$$
\Psi_1(u)=\lim_{N\ra +\infty}{e^{u^2(\log N)}
\expect(e^{iu\log |\det(1-X_N)|^2})},
$$
where $X_N$ is a Haar-distributed $U(N)$-valued random
variable. Secondly (see~\cite[4.1]{jkn}), we have
$$
\Psi_2(u)=\lim_{N\ra+\infty}{
e^{u^2(\log (e^{\gamma}\log N))}\expect(e^{iuL_N})}
$$
where
$$
L_N=\sum_{p\leq N}{\log
  \Bigl|1-\frac{e^{i\theta_p}}{\sqrt{p}}\Bigr|^2},
$$
for any sequence $(\theta_p)_{p\leq N}$ of independent random
  variables, uniformly distributed on $[0,1]$.
\par
\begin{rem}
  Note in passing that for fixed $p$, the $p$-th component of the
  Euler product of $\zeta(1/2+iU)$, for $U$ uniformly distributed on
  $[0,T]$, converges in law to $(1-e^{i\theta_p}p^{-1/2})^{-1}$ as
  $T\ra +\infty$.
\end{rem}
\par
We now prove that the Euler product $\Phi_2$ (like $\Psi_2$)
corresponds to mod-Poisson convergence for a natural asymptotic
probabilistic model of primes, and that $\Phi_1$ (like $\Psi_1$) comes
from a model of group-theoretic origin.\footnote{\ Since a product of
  two limiting functions for mod-Poisson convergence is clearly
  another such limiting function, we also recover without arithmetic
  the fact that the limiting function $\Phi(u)$ arises from
  mod-Poisson convergence.}
\par
We start with the Euler product, where the computation was already
described in~\cite[\S 4.3]{jkn}: we have
$$
\Phi_2(u)=\lim_{y\ra +\infty}{\prod_{p\leq y}{
    \Bigl(1-\frac{1}{p}\Bigr)^{e^{iu}-1}
\Bigl(1-\frac{1}{p}\Bigr)
\Bigl(1+\frac{e^{iu}}{p-1}\Bigr)
}},
$$
and by isolating the first term, it follows that
\begin{align*}
\Phi_2(u)&=\lim_{y\ra +\infty}{
\exp((1-e^{iu})\lambda_y)
\prod_{p\leq y}{\Bigl(1-\frac{1}{p}+\frac{1}{p}e^{iu}\Bigr)
}}\\
&=\lim_{y\ra +\infty}{\E(e^{iuP_{\lambda_y}})^{-1}}
\E(e^{iuZ'_y})
\end{align*}
where
$$
\lambda_y=\sum_{p\leq y}{\log \Bigl(\frac{1}{1-p^{-1}}\Bigr)}=
\sum_{\stacksum{p\leq y}{k\geq 1}}{
\frac{1}{kp^k}}=\log\log y+\kappa+o(1),
$$
as $y\ra +\infty$, for some real constant $\kappa$ (see,
e.g.,~\cite[\S 22.8]{hardy-wright}), and
\begin{equation}\label{eq-zy}
Z'_y=\sum_{p\leq y}{B'_{p}}
\end{equation}
is a sum of independent Bernoulli random variables with parameter
$1/p$:
$$
\proba(B'_{p}=1)=\frac{1}{p},\quad\quad
\proba(B'_{p}=0)=1-\frac{1}{p}.
$$
\par
We note that this is a particular case of
Proposition~\ref{prop-poisson-convergence}, and that (as expected) the
parameters of these Bernoulli laws correspond exactly to the
``intuitive'' probability that an integer $n$ be divisible by $p$, or
equivalently, the Bernoulli variable $B'_p$ is the limit in law as
$N\ra +\infty$ of the random variables defined as the indicator of a
uniformly chosen integer $n\leq N$ being divisible by $p$; the
independence of the $B'_p$ corresponds for instance to the formal
(algebraic) independence of the divisibility by distinct primes given,
e.g., by the Chinese Remainder Theorem.
\par
As in the case of the Riemann zeta function, we also note that the
independent model fails to capture the truth on the distribution of
$\omega(n)$, the extent of this failure being measured, in some sense,
by the factor $\Phi_1(u)$. Because
$$
\frac{Z'_y-\log \log y}{\sqrt{\log\log y}}\convlaw \mathcal{N}(0,1),
$$
this discrepancy between the independent model and the arithmetic
truth is invisible at the level of the normalized convergence in
distribution (as it is for $\log|\zeta(1/2+it)|$, by Selberg's Central
Limit Theorem, hiding the Random Matrix Model).
\par
Now we consider the first factor
$\Phi_1(u)=\Gamma(e^{iu}+1)^{-1}$. Again, in~\cite[\S 4.3]{jkn}, we
appealed to the formula
$$
\frac{1}{\Gamma(e^{iu}+1)}=\prod_{k\geq 1}{
  \Bigl(1+\frac{e^{iu}}{k}\Bigr)
  \Bigl(1+\frac{1}{k}\Bigr)^{-e^{iu}}
}
$$
for $u\in\Rr$ (see~\cite[12.11]{ww}) to compute
\begin{align*}
  \Phi_1(u)&=\lim_{N\ra +\infty}{ \prod_{k\leq N}{
      \Bigl(1+\frac{1}{k}\Bigr)^{1-e^{iu}}
      \Bigl(1+\frac{1}{k}\Bigr)^{-1} \Bigl(1+\frac{e^{iu}}{k}\Bigr) }
  }\\
  &= \lim_{N\ra +\infty}\exp(\lambda_N(1-e^{iu}))\prod_{k\leq N}{
    \Bigl(1+\frac{1}{k}\Bigr)^{-1} \Bigl(1+\frac{e^{iu}}{k}\Bigr)
  }\\
  &=\lim_{N\ra +\infty}\exp(\lambda_N(1-e^{iu})) \expect(e^{iu Z_N}),
\end{align*}
where 
$$
\lambda_N=\sum_{1\leq k\leq N}{\log (1+k^{-1})}=\log (N+1),
$$
and $Z_N$ is the sum 
$$
Z_N=B_{1}+B_{2}+\cdots +B_{N},
$$
with $B_{k}$ denoting independent Bernoulli random variables with
distribution
$$
\proba(B_{k}=1)=1-\frac{1}{1+\frac{1}{k}}=\frac{1}{k+1},\quad\quad
\proba(B_{k}=0)=\frac{1}{1+\frac{1}{k}}=\frac{k}{k+1}.
$$
\par
The group-theoretic interpretation of this distribution is very
suggestive: indeed, it is the distribution of the random variable
$\varpi(\sigma_{N+1})-1$, where $\sigma_{N+1}\in \sym_{N+1}$ is
distributed according to the uniform measure on the symmetric group,
and we recall that $\varpi(\sigma)$ is the number of cycles of a
permutation. In other words, we have
\begin{equation}\label{eq-char-permut}
\expect(e^{iu\varpi(\sigma_N)})=\prod_{1\leq j\leq N}{
\Bigl(1-\frac{1}{j}+\frac{e^{iu}}{j}\Bigr)},
\end{equation}
as proved, e.g., in~\cite[\S 4.6]{abt}; note that this is not obvious,
and the decomposition as a sum of independent random variables is due
to Feller, and is explained in~\cite[p. 16]{abt}.
\par
So we see -- and this gives another example of natural mod-Poisson
convergence -- that these random variables have mod-Poisson convergence
with parameters $\log N$, and limiting function $1/\Gamma(e^{iu})$:
\begin{equation}\label{eq-mod-poisson-permut}
\lim_{N\ra +\infty}\exp((\log N)(1-e^{iu}))\expect(e^{iu\varpi(\sigma_N)})=
\frac{1}{\Gamma(e^{iu})}.
\end{equation}
\par
For further reference, we state a more precise version, which follows
from~(\ref{eq-char-permut}):
\begin{equation}\label{eq-explicit-permut}
\expect(e^{iu\varpi(\sigma_N)})=
\frac{1}{\Gamma(e^{iu})}\exp((\log N)(e^{iu}-1))\erreurm{\frac{1}{N}},
\end{equation}
locally uniformly for $u\in\Rr$. Note that this includes the special
case $u=(2k+1)\pi$ where 
$$
\expect(e^{iu\varpi(\sigma_N)})=
\frac{1}{\Gamma(e^{iu})}=0.
$$
\par
\par
This explanation of the ``transcendental'' factor $1/\Gamma(e^{iu}+1)$
is particularly convincing because of well-known and well-studied
analogies between the cycle structure of random permutations and the
factorization of integers (see, e.g., the discussion in~\cite[\S
1.2]{abt} or the entertaining survey~\cite{granville}). Its origin
in~\cite[4.3]{jkn} is, however, not very enlightening: the Gamma
function appears universally in the Delange-Selberg method in a way
which may seem to be coincidental and unrelated to any group-theoretic
structure (see, e.g.,~\cite[\S 5.2]{tenenbaum} where it originates in
a representation of $1/\Gamma(z)$ as a contour integral of Hankel
type).

\section{The analogy deepens}\label{sec-finite-fields}

The discussion of the previous section is already interesting, but it
becomes (to our mind) even more intriguing after one notes how the
analogy can be extended by including consideration of function
field situations, as in the work of Katz-Sarnak~\cite{katzsarnak}.
\par
Let $\Fp_q$ be a finite field with $q=p^n$ elements, with $n\geq 1$
and $p$ prime.  For a polynomial $f\in \Fp_q[X]$, let
$$
\omega(f)=\omega_q(f)=|\{\irred\in \Fp_q[X]\,\mid\, \irred\text{ is
  irreducible monic and divides } f\}|
$$
be the analogue of the number of prime factors of an integer (we will
usually drop the subscript $q$).
\par
We consider the statistic behavior of this function under two types of
limits: (i) either $q$ is replaced by $q^m$, $m\ra +\infty$, and $f$
is assumed to range over monic polynomials of fixed degree $d\geq 1$
in $\Fp_{q^m}[X]$; or (ii) $q$ is fixed, and $f$ is assumed to range
over monic polynomials of degree $d\ra +\infty$ in $\Fp_q[X]$.
\par
The first limit, of fixed degree and increasing base field, is similar
to the one considered by Katz and Sarnak for the distribution of zeros
of families of $L$-functions over finite fields~\cite{katzsarnak}. And
the parallel is quite precise as far as the group-theoretic situation
goes. Indeed, recall that the crucial ingredient in their work is that
the Frobenius automorphism provides in a natural way a ``random
matrix'' for a given $L$-function, the characteristic polynomial of
which provides a spectral interpretation of the zeros (see,
e.g.,~\cite[\S 4.2]{jkn} for a partial, down-to-earth, summary).
\par
In our case, let us assume first that $f\in \Fp_{q}[X]$ is
squarefree. Let $K_f$ denote the splitting field of $f$, i.e., the
extension field of $\Fp_{q}$ generated by the $d$ roots of $f$, and
let $F_f$ denote the Frobenius automorphism $x\mapsto x^q$ of
$K_f$. This automorphism permutes the roots of $f$, which all lie in
$K_f$, and after enumerating them, leads to an element of $\sym_d$,
denoted $F_f$. This depends on the enumeration of the roots, but the
conjugacy class $F_f^{\sharp}\in\sym_d^{\sharp}$ is well-defined.
\par
Now, by the very definition, we have
\begin{equation}\label{eq-ell-om}
\omega(f)=\varpi(F_f^{\sharp}),
\end{equation}
which can be seen as the (very simple) analogue of the spectral
interpretation of an $L$-function as the characteristic polynomial of
the Frobenius endomorphism.

\begin{rem}
\label{rm-scheme}
We can come even closer to the Katz-Sarnak setting of families of
$L$-functions. Consider, in scheme-theoretic language,\footnote{\
  Readers unfamiliar with this language can skip this remark, which
  will not be used, except to state Theorem~\ref{th-katz-sarnak}
  below.}  the (very simple!) family of zeta functions of the
zero-dimensional schemes $X_f=\spec(\Fp_q[X]/(f))$, i.e., the
varieties over $\Fp_q$ with equation $f(x)=0$. These zeta functions
are defined by either of the following two formulas:
$$
Z(X_f)=\prod_{x\in |X_f|}{(1-T^{\deg(x)})^{-1}}=\exp\Bigl(
\sum_{m\geq 1}{\frac{|X_f(\Fp_{q^m})|T^m}{m}}
\Bigr),
$$
where $|X_f|$ is the set of closed points of $X_f$. Since these
correspond naturally to irreducible factors of $f$ (without
multiplicity), it follows that
$$
Z(X_f)=\prod_{\irred\mid f}{(1-T^{\deg(\irred)})^{-1}},
$$
and hence, if $f$ is squarefree, a higher-level version
of~(\ref{eq-ell-om}) is the ``spectral interpretation''
\begin{equation}\label{eq-zeta-spectral}
Z(X_f)=\det(1-F_fT|H^0_c(\bar{X}_f,\mathbf{Q}_{\ell}))^{-1}=
\det(1-\rho(F_f)T)^{-1}
\end{equation}
where $F_f$ is still the Frobenius automorphism,
$H^0_c(\bar{X}_f,\mathbf{Q}_{\ell})$ is simply isomorphic with
$\Qq_{\ell}^{\deg(f)}$ (the variety over the algebraic closure has
$\deg(f)$ connected components, which are points), and $\rho$ is the
natural faithful representation of $\mathfrak{S}_{\deg(f)}$ in
$U(\deg(f),\Cc)$ by permutation matrices, since this is quite clearly
how $F_f$ acts on the \'etale cohomology space.
\par
Looking at the order of the pole of $Z(X_f)$ at $T=1$, we
recover~(\ref{eq-ell-om}). In particular, the generalizations of the
Erd\H{o}s-Kac Theorem that we will prove in the next section can be
interpreted as describing the limiting statistical behavior, in
mod-Poisson sense, of the order of the pole of those zeta functions as
the degree $\deg(f)$ tends to infinity (see
Theorem~\ref{th-katz-sarnak}). It is truly a zero-dimensional version
of the Katz-Sarnak problematic for growing conductor. (Note that this
interpretation also suggests to look at other distribution statistics
of these zeta functions, and we hope to come back to this).
\end{rem}
\par
\medskip
\par 
The relation~(\ref{eq-ell-om}) (or~(\ref{eq-zeta-spectral})) explains
the existence of a link between the number of irreducible factors of
polynomials and the number of cycles of permutations.  Indeed, the
other essential number-theoretic ingredient for Katz and Sarnak is
Deligne's Equidistribution Theorem, which shows that the matrices
given by the Frobenius, \emph{in the limit under consideration} where
$q$ is replaced by $q^m$, $m\ra +\infty$, become equidistributed in a
certain monodromy group. Here we have, exactly similarly, the
following well-known:
\par
\medskip
\par
\textbf{Fact.} In the limit of fixed $d$ and $m\rightarrow +\infty$,
for $f$ uniformly chosen among monic squarefree polynomials of degree
$d$ in $\Fp_{q^m}[X]$, the conjugacy classes $F_f^{\sharp}$ become
uniformly distributed in $\sym_d^{\sharp}$ for the natural (Haar)
measure.
\par
\medskip
This fact is easily proved from the well-known Gauss-Dedekind
formula
$$
\Pi_q(d)=
\sum_{\deg(\irred)=d}{1}=
\frac{1}{d}\sum_{\delta\mid d}{\mu(\delta)q^{d/\delta}}=
\frac{q^{d}}{d}+O(q^{d/2})
$$
for the number of irreducible monic polynomials of degree $d$ with
coefficients in $\Fp_q$, and it is a ``baby'' analogue of Deligne's
Equidistribution Theorem.\footnote{\ Indeed, it could be proved using
  the Chebotarev density theorem, which is a special case of Deligne's
  theorem.} Hence, we obtain
$$
\omega(f)\convlaw \varpi(\sigma_{d}),
$$
as $m\ra +\infty$, where $f$ is distributed uniformly among monic
polynomials of degree $d$ in $\Fp_{q^m}[X]$, and $\sigma_d$ is
distributed uniformly among $\sym_d$.
\par
The second limit, where the base field $\Fp_q$ is fixed and the degree
$d$ grows, is analogue of the problematic situation of families of
curves of increasing genus over a fixed finite field (see the
discussion in~\cite[p. 12]{katzsarnak}), and -- for our purposes -- of
the distribution of the number of prime divisors of integers, which we
discussed in the previous section. In the next section, we prove a
mod-Poisson form of the Erd\H{o}s-Kac theorem in $\Fp_q[X]$ (the
Central Limit version being a standard result, essentially due to
M. Car, and apparently stated first by Flajolet and Soria~\cite[\S 3,
Cor. 1]{flajolet-soria}; see also the recent quick derivation by
R. Rhoades~\cite{rhoades}).

\begin{rem}
  One may extend the conjugacy class $F_f^{\sharp}\in \sym_d^{\sharp}$
  to all $f\in\Fp_q[X]$ of degree $d$, in the following directly
  combinatorial way (which hides the Frobenius aspect): $F_f^{\sharp}$
  is the conjugacy class of permutations with as many disjoint
  $j$-cycles, $1\leq j\leq d$, as there are irreducible factors of $f$
  of degree $j$. However, the relation
  $\omega(f)=\varpi(F_f^{\sharp})$ does \emph{not} extend to this
  case, since multiple factors are not counted by $\omega$. However,
  we have $\Omega(f)=\varpi(F_f^{\sharp})$, where $\Omega(f)$ is the
  number of irreducible factors counted with multiplicity.
\end{rem}

\section{Mod-Poisson convergence for the number of irreducible factors
  of a polynomial}\label{sec-main}

In this section, we state and prove the mod-Poisson form of the
analogue of the Erd\H{o}s-Kac Theorem for polynomials over finite
fields, trying to bring to the fore the probabilistic structure
suggested in the previous section. 

\begin{thm}\label{th-main}
  Let $q\not=1$ be a power of a prime $p$, and let $\omega(f)$ denote
  as before the number of monic irreducible polynomials dividing $f\in
  \Fp_q[X]$. Write $|g|=q^{\deg(g)}=|\Fp_q[X]/(g)|$ for any non-zero
  $g\in \Fp_q[X]$.
\par
For any $u\in\RR$, we have
\begin{equation}\label{eq-main}
\lim_{d\ra +\infty}\frac{\exp((1-e^{iu})\log d)}{q^d}
\sum_{\deg(f)=d}
{e^{iu (\omega(f)-1)}}=
\tilde{\Phi}_1(u)\tilde{\Phi}_2(u),
\end{equation}
where 
\begin{equation}\label{eq-phi1}
\tilde{\Phi}_1(u)=\frac{1}{\Gamma(e^{iu}+1)}
\end{equation}
and
\begin{equation}\label{eq-phi2}
\tilde{\Phi}_2(u)=\prod_{\irred}{
\Bigl(1-\frac{1}{\nirred}\Bigr)^{e^{iu}}
\Bigl(1+\frac{e^{iu}}{\nirred-1}\Bigr) },
\end{equation}
the product running over all monic irreducible polynomials $\irred\in
\Fp_q[X]$ and the sum over all monic polynomials $f\in \Fp_q[X]$ with
degree $\deg(f)=d$. Moreover, the convergence is uniform.
\end{thm}

\begin{rem}
Note the similarity of the shape of the limiting function
with that in~(\ref{eq-phi-omega}) and the conjecture for $\zeta(1/2+it)$, in
particular the fact that the group-theoretic term is the same as for
$\omega(n)$, while the Euler product is a direct transcription in
$\Fp_q[X]$ of the earlier $\Phi_2$. 
\end{rem}

\begin{rem}
This can be rephrased, according to Remark~\ref{rm-scheme}, in the
following manner which illustrates the analogy with the Katz-Sarnak
philosophy: 
\begin{thm}\label{th-katz-sarnak}
  Let $q\not=1$ be a power of a prime. For any $f\in\Fp_q[X]$, monic
  of degree $\geq 1$, let $X_f$ be the zero-dimensional scheme
  $\spec(\Fp_q[X]/(f))$, let $Z(X_f)\in \Qq(T)$ denote its zeta
  function and let $r(X_f)\geq 0$ denote the order of the pole of
  $Z(X_f)$ at $T=1$. Then for any $u\in\RR$, we have
$$
\lim_{d\ra +\infty}\frac{\exp((1-e^{iu})\log d)}{q^d}
\sum_{\deg(f)=d}
{e^{iu r(f)}}=
e^{-iu}\tilde{\Phi}_1(-u)\tilde{\Phi}_2(-u),
$$
with notation as before.
\end{thm}

The only thing to note here is that if $f$ is not squarefree, the
scheme $X_f$ is not reduced; the induced reduced scheme is
$X_{f^{\flat}}$, where $f^{\flat}$ is the (squarefree) product of the
distinct monic irreducible factors dividing $f$. Then
$Z(X_f)=Z(X_{f^{\flat}})$, and we have
$$
-r(f)=\ord_{T=1}Z(X_f)=\ord_{T=1}Z(X_{f^{\flat}})=-r(f^{\flat})=
\omega(f^{\flat})=\omega(f),
$$
so the two theorems are indeed equivalent.
\end{rem}

\begin{rem}
One can also prove by the same method the following two variants,
where we restrict attention to squarefree polynomials, or we consider
irreducible factors with multiplicity. First, we have
$$
\frac{e^{(1-e^{iu})\log d}}{q^d}
\sumb_{\deg(f)=d}
{e^{iu (\omega(f)-1)}}\ra 
\frac{1}{\Gamma(1+e^{iu})}
\prod_{\irred}{
\Bigl(1-\frac{1}{\nirred}\Bigr)^{e^{iu}}
\Bigl(1+\frac{e^{iu}}{\nirred}\Bigr) },
$$
where the sum $\sumb$ runs over all squarefree monic polynomials $f\in
\Fp_q[X]$ with degree $\deg(f)=d$. Next, we have
$$
\frac{e^{(1-e^{iu})\log d}}{q^d}
\sumb_{\deg(f)=d}
{e^{iu (\Omega(f)-1)}}\ra 
\frac{1}{\Gamma(1+e^{iu})}
\prod_{\irred}{
\frac{(1-\nirred^{-1})^{e^{iu}}}
{1-e^{iu}/\nirred}}.
$$
\end{rem}

We now come to the proof. The idea we want to highlight -- the source
of the splitting of the limiting function in two parts of distinct
probabilistic origin -- is to first separate the irreducible factors
of ``small'' degree and those of ``large'' degree (which is fairly
classical), and then observe that an equidistribution theorem allows
us to perform a transfer of the contribution of large factors to the
corresponding average over random permutations, conditioned to not
have small cycle lengths. This will explain the factor
$\tilde{\Phi}_1$ corresponding to the cycle length of random
permutations. Note that shorter arguments are definitely available,
using analogues of the Delange-Selberg method used in~\cite{jkn}
(see~\cite[\S 2, Th. 1]{flajolet-soria}), but this hides again the
mixture of probabilistic models involved.
\par
Interestingly, the small and larger irreducible factors are \emph{not}
exactly independent. But the dependency is (essentially) perfectly
compensated by the effect of the conditioning at the level of random
permutations. Why this is so may be the last little mystery in the
computation, which is otherwise very enlightening.
\par
We set up some notation first: for $f\in \Fp_q[X]$, we let $\degp{f}$
(resp.  $\degm{f}$) denote the largest (resp., smallest) degree of an
irreducible factor $\irred\mid f$; correspondingly, for a permutation
$\sigma\in\sym_d$, we denote by $\ellp{\sigma}$
(resp. $\ellm{\sigma}$) the largest (resp. smallest) length of a cycle
occurring in the decomposition of $\sigma$.
\par
Henceforth, by convention, any sum involving polynomials $f$, $g$,
$h$, etc, is assumed to restrict to monic polynomials, and any sum or
product involving $\irred$ is restricted to monic irreducible
polynomials.
\par
The next lemma summarizes some simple properties, and the important
equidistribution property we need.

\begin{lem}\label{lm-distrib}
With notation as above, we have:
\par
\emph{(1)} For all $d\geq 1$, we have
$$
\frac{1}{q^d}\sum_{\deg(\irred)=d}{1}=\frac{1}{d}+O(q^{-d/2}).
$$
\par
\emph{(2)} For all $d\geq 1$, we have
\begin{gather}\label{eq-upper-mertens}
\prod_{\deg(\irred)\leq d}{\Bigl(1+\frac{1}{\nirred-1}\Bigr)}\ll d,
\\
\label{eq-mertens}
\prod_{\deg(\irred)\leq d}{\Bigl(1-\frac{1}{\nirred}\Bigr)}=
\exp\Bigl(-\sum_{1\leq j\leq d}{\frac{1}{j}}\Bigr)
\erreurm{\frac{1}{d}}.
\end{gather}
\par
\emph{(3)} For any $d\geq 1$ and any fixed permutation
$\sigma\in\sym_d$, we have
\begin{equation}\label{eq-quant-equid}
  \frac{1}{q^d}\sumb_{\stacksum{\deg(f)=d}{F_f^{\sharp}=\sigma^{\sharp}}}{1}
  =\proba(\sigma_d=\sigma)\erreurm{ \frac{d}{q^{\ellm{\sigma}/2}}},
\end{equation}
where the conjugacy class $F_f^{\sharp}\in\sym_d^{\sharp}$ is defined
in the previous section, $\sigma_d$ is a uniformly chosen random
permutation in $\sym_d$ and $\sumb$ restricts the sum to squarefree
polynomials. 
\par
In all estimates, the last under the assumption
$q^{\ellm{\sigma}/2}\geq d$, the implied constants are absolute,
except that in~\emph{(\ref{eq-upper-mertens})}, the implied constant
may depend on $q$.
\end{lem}

\begin{proof}
  The first statement has already been
  recalled. For~(\ref{eq-upper-mertens}), we have
\begin{align*}
\prod_{\deg(\irred)\leq d}{\Bigl(1+\frac{1}{\nirred-1}\Bigr)} 
&\leq
\exp\Bigl(\sum_{1\leq j\leq d} {\frac{\Pi_q(j)}{q^j-1}}\Bigr)\\
&=\exp\Bigl(\sum_{2\leq j\leq d}{\frac{1}{j}}
+O\Bigl(\sum_{1\leq j\leq d}{\frac{q^{j/2}}{q^j-1}}\Bigr)\Bigr)
\ll d,
\end{align*}
for $d\geq 1$, with an implied constant depending on $q$.
\par
For~(\ref{eq-mertens}), which is the analogue for $\Fp_q[T]$ of the
classical Mertens estimate, we refer, e.g., to~\cite{rosen}, where it
is proved in the form
$$
\prod_{\deg(\irred)\leq d}{\Bigl(1-\frac{1}{\nirred}\Bigr)}=
\frac{e^{-\gamma}}{d}\Bigl(1+O\Bigl(\frac{1}{d}\Bigr)\Bigr)
$$
for $d\geq 1$, $\gamma$ being the Euler constant; since 
$$
\sum_{1\leq j\leq d}{\frac{1}{j}}=\log d+\gamma+
O\Bigl(\frac{1}{d}\Bigr),
$$
we get the stated result. We emphasize the fact that the asymptotic of
the product in~(\ref{eq-mertens}) is independent of $q$ (and is the
same as for the usual Mertens formula for prime numbers), since this
may seem surprising at first sight. This is explained by the relation
with random permutations, and in fact, in Remark~\ref{rm-mertens}
below, we explain how our argument leads to a much sharper
estimate~(\ref{eq-mertens2}) for the error term in~(\ref{eq-mertens}).
\par
Finally, for the third statement, if $\sigma$ is a product of $r_j$
disjoint $j$-cycles for $1\leq j\leq d$, we first recall the standard
formula that
\begin{equation}\label{eq-card-conj}
\proba(\sigma_d=\sigma)=\prod_{1\leq j\leq d}{\frac{1}{j^{r_j}r_j!}},
\end{equation}
and we observe that the product can be made to range over
$\ellm{\sigma}\leq j\leq d$, since the terms $j<\ellm{\sigma}$ have
$r_j=0$ by definition. Using this observation, we have by simple
counting
$$
\sumb_{\stacksum{\deg(f)=d}{F_f^{\sharp}=\sigma^{\sharp}}}{1}
=\prod_{\ellm{\sigma}\leq j\leq d}{ \binom{\Pi_q(j)}{r_j}}.
$$
\par
Furthermore, for any $r$ and $j\geq 1$ such that $r<q^{j/2}$ and
$j\leq r$, we have
\begin{align*}
\binom{\Pi_q(j)}{r}&=\frac{1}{r!}\Pi_q(j)(\Pi_q(j)-1)\cdots (\Pi_q(j)-r+1)
\\
&=\frac{1}{r!}\Bigl(\frac{q^j}{j}+O(q^{-j/2})\Bigr)^r
=\frac{q^{jr}}{j^{r}r!}(1+O(rq^{-j/2}))^{r},
\end{align*}
by the first part of the lemma. Combining the two formulas, we get
\begin{align*}
\frac{1}{q^d}\sumb_{\stacksum{\deg(f)=d}{F_f^{\sharp}=\sigma^{\sharp}}}{1}
&=q^{-d}
\prod_{\ellm{\sigma}\leq j\leq d}
\frac{q^{jr_j}}{j^{r_j}r_j!}(1+O(dq^{-j/2}))^{r_j}\\
&=
\prod_{\ellm{\sigma}\leq j\leq  d}{\frac{1}{r_j!j^{r_j}}
\Bigl(1+O(dq^{-j/2})\Bigr)^{r_j}}\\
&=\proba(\sigma_d=\sigma)\prod_{\ellm{\sigma}\leq j\leq  d}
{\Bigl(1+O(dq^{-j/2})\Bigr)^{r_j}}
\end{align*}
and this immediately gives the conclusion since the implied constant
in the formula for $\Pi_q(j)$ is at most $1$. 
\end{proof}

Part (3) of this lemma means that, as long as we consider permutations
$\sigma\in\sym_d$ with no short cycle, so that
$$
d=o(q^{\ellm{\sigma}/2}),
$$
there is strong quantitative equidistribution of the conjugacy class
$F_f^{\sharp}$ among all conjugacy classes in $\sym_d$.
\par
Thus, to compare the distribution of polynomials and that of
permutations, it is natural to introduce a parameter $b$, $0\leq b\leq
d$, to be specified later, and to first write any monic polynomial $f$
of degree $d$ as $f=gh$, where the monic polynomials $g$ and $h$ are
uniquely determined by
\begin{equation}\label{eq-split-poly}
\degp{g}\leq b,\quad\quad \degm{h}>b
\end{equation}
(i.e., $g$ contains the small factors, and $h$ the large ones; they
correspond to ``friable'' and ``sifted'' integers in classical
analytic number theory). One can expect, by the above, that if $b$ is
such that $q^{b/2}$ is large enough compared with $d$, the
distribution of $h$ will reflect that of permutations without cycles
of length $\leq b$. And the contribution of small factors should (and
will) be comparable with the independent model for divisibility of
polynomials by irreducible ones.
\par
We now start the proof of Theorem~\ref{th-main} along these
lines, trying to evaluate 
$$
\frac{1}{q^d}
\sum_{\deg(f)=d}
{e^{iu \omega(f)}}
$$
\par
Writing $f=gh$, where $g$ and $h$ satisfy~(\ref{eq-split-poly}) as
above, we have $\omega(f)=\omega(g)+\omega(h)$ since $g$ and $h$ are
coprime, and hence
$$
\frac{1}{q^d}
\sum_{\deg(f)=d}
{e^{iu \omega(f)}}=
\sum_{\stacksum{\deg(g)\leq d}{\degp{g}\leq b}}
{\frac{e^{iu\omega(g)}}{|g|}
T(d-\deg(g),b)},
$$
where we define
$$
T(d,b)=\frac{1}{q^d}\sum_{\stacksum{\deg(f)=d}{\degm{f}>b}}{e^{iu\omega(f)}}.
$$
\par
Denote further
$$
R(d,b)=\sum_{\stacksum{\deg(g)>d}{\degp{g}\leq b}}
{\frac{1}{|g|}},\quad\quad
S(d,b)=\sumb_{\stacksum{\deg(g)\leq d}{\degp{g}\leq b}}
{\frac{1}{|g|}}.
$$
\par
Noting that $|T(d,b)|\leq 1$ for all $b$ and $d$, and splitting the
sum over $g$ according as to whether $\deg(g)\leq \sqrt{d}$ or
$\deg(g)>\sqrt{d}$, we get
\begin{align}
\frac{1}{q^d}
\sum_{\deg(f)=d}
{e^{iu \omega(f)}}&=
\sum_{\stacksum{\deg(g)\leq \sqrt{d}}{\degp{g}\leq b}}
{\frac{e^{iu\omega(g)}}{|g|}
T(d-\deg(g),b)}+O(R(\sqrt{d},b))\nonumber\\
&=S_1+O(R(\sqrt{d},b)),\text{ say}.\label{eq-first}
\end{align}
\par
The next step, which is were random permutations will come into play,
will be to evaluate $T(d,b)$ asymptotically in suitable ranges.

\begin{prop}\label{prop-td}
With notation as before, we have
\begin{multline}
T(d,b)=\exp\Bigl(-e^{iu}\sum_{j=1}^b{\frac{1}{j}}\Bigr) \expect(e^{iu
  \varpi(\sigma_d)})+ \\
O\Bigl(|\expect(e^{iu
  \varpi(\sigma_d)})|b^2d^{-1}+dq^{-b/2} +b^3(\log d)^{1/2}d^{-2}\Bigr),
\end{multline}
with an absolute implied constant, in the range
\begin{equation}\label{eq-valid}
  q^{b/2}\geq d,\quad b\leq d.
\end{equation}
\end{prop}

\begin{proof}
Before introducting permutations, we separate the contribution of
squarefree and non-squarefree polynomials in $T(d,b)$ (the intuition
being that non-squarefree ones should be much sparser than for all
polynomials because of the imposed divisibility only by large
factors):
$$
T(d,b)=T^{\flat}(d,b)+T^{\sharp}(d,b)
$$
where
$$
T^{\flat}(d,b)=\frac{1}{q^d}
\sumb_{\stacksum{\deg(f)=d}{\degm{f}>b}}{e^{iu\omega(f)}},
$$
and $T^{\sharp}(d,b)$ is the complementary term. We then estimate the
latter by
\begin{align*}
|T^{\sharp}(d,b)|& 
\leq \sum_{b\leq \deg(g)\leq d/2}{
\frac{1}{q^d}
\sum_{\stacksum{\deg(f)=d}{g^2\mid f}}{1}}
\\
&=
\sum_{b\leq \deg(g)\leq d/2}{
\frac{1}{q^d}
\sum_{\deg(f)=d-2\deg(g)}{1}
}\\
&\leq\sum_{\deg(g)\geq b}{\frac{1}{q^{2\deg(g)}}}
\ll \frac{1}{q^b}.
\end{align*}
\par
We can now introduce permutations through the association $f\mapsto
F_f^{\sharp}$ sending a squarefree polynomial to its associated cycle
type. Using~(\ref{eq-ell-om}), we obtain
$$
T^{\flat}(d,b)=\sum_{\stacksum{\sigma\in\sym_d}{\ellm{\sigma}>b}}{
  e^{iu\varpi(\sigma)}
  \frac{1}{q^d}\sumb_{\stacksum{\deg(f)=d}{F_f^{\sharp}=\sigma^{\sharp}}}{1}
},
$$
which is now a sum over permutations without small cycles. Using the
third statement of Lemma~\ref{lm-distrib}, we derive
\begin{align*}
  T^{\flat}(d,b)&=\sum_{\stacksum{\sigma\in\sym_d}
    {\ellm{\sigma}>b}}{ e^{iu\varpi(\sigma)}\proba(\sigma_d=\sigma)
    \erreurm{\frac{d}{q^{b/2}}}}\\
  &=\expect(e^{iu\varpi(\sigma_d)}\charfun_{\ellm{\sigma_d}>b})
  +O\Bigl(\proba(\ellm{\sigma_d}>b)dq^{-b/2}\Bigr),
\end{align*}
with an absolute implied constant if $q^{b/2}\geq d$.
\par
Thus the problem is reduced to one about random permutations. Using
Proposition~\ref{pr-permut} below with $\eps=1$, the proof is
finished.
\end{proof}

Now recall that the characteristic function
$\expect(e^{iu\varpi(\sigma_d)})$ is explicitly known
from~(\ref{eq-char-permut}). This formula,
or~(\ref{eq-explicit-permut}), implies in particular that we have
\begin{equation}\label{eq-unif-char}
\expect(e^{iu\varpi(\sigma_{d-j})})
=\expect(e^{iu \varpi(\sigma_d)})\erreurm{\frac{j}{d}}.
\end{equation}
\par
Then, inserting the formula of Proposition~\ref{prop-td} in the first
term $S_1$ of~(\ref{eq-first}), and using this formula, we obtain in
the range of validity~(\ref{eq-valid}) that
$$
S_1=\exp\Bigl(-e^{iu}
\sum_{j=1}^b{\frac{1}{j}}\Bigr)\expect(e^{iu \varpi(\sigma_{d})})
   \sum_{\stacksum{\deg(g)\leq \sqrt{d}}{\degp{g}\leq b}}
   {\frac{e^{iu\omega(g)}}{|g|}}+R
$$
where, after some computations, we find that
$$
R\ll (|\expect(e^{iu\varpi(\sigma_d)})|b^2d^{-1}
+dq^{-b/2}+b^3(\log d)^{1/2}d^{-2})S(\sqrt{d},b),
$$
with an absolute implied constant.
\par
Extending the sum in the main term, we get
$$
S_1=M+R_1,
$$
where
\begin{gather*}
M=
\expect(e^{iu \varpi(\sigma_{d})})
\exp\Bigl(-e^{iu}\sum_{j=1}^b{\frac{1}{j}}\Bigr)
\sum_{\degp{g}\leq b}{\frac{e^{iu\omega(g)}}{|g|}},
\\
R_1\ll bR(\sqrt{d},b)+
\Bigl(|\expect(e^{iu\varpi(\sigma_d)})|\frac{b^2}{d}
+\frac{d}{q^{b/2}}+
\frac{b^3(\log d)^{1/2}}{d^{2}}\Bigr)S(\sqrt{d},b).
\end{gather*}
\par
Now, we can finally apply~(\ref{eq-mertens}) and multiplicativity in
the sum over $g$ in $M$, to see that
\begin{align*}
M=\expect(e^{iu \varpi(\sigma_{d})})\prod_{\deg(\irred)\leq b}{
\Bigl(1-\frac{1}{\nirred}\Bigr)^{e^{iu}}
\Bigl(1+\frac{e^{iu}}{\nirred-1}\Bigr)
}
\Bigl(1+O\Bigl(\frac{1}{b}\Bigr)\Bigr)
\end{align*}
and hence, by the mod-Poisson convergence of $\varpi(\sigma_d)$ and
the absolute convergence of the Euler product extended to infinity, we
have
$$
\lim_{d,b\ra +\infty}\exp((\log d)(1-e^{iu}))M
=\tilde{\Phi}_1(u)\tilde{\Phi}_2(u),
$$
uniformly for $u\in \Rr$.
\par
There remain to consider the error terms to conclude the proof of
Theorem~\ref{th-main}. We select $b=(\log d)^2\ra +\infty$;
then~(\ref{eq-valid}) holds for all $d\geq d_0(q)$, and hence the
previous estimates are valid and we must now show that
$$
\exp((\log d)(1-e^{iu}))R(\sqrt{d},b)\ra 0,
\quad
\exp((\log d)(1-e^{iu}))R_1\ra 0
$$
(the first desideratum coming from~(\ref{eq-first})).
\par
Note that $|\exp((\log d)(1-e^{iu}))|\leq d^2$. Now we claim that
\begin{align}\label{eq-rankin}
R(d,b)&\ll b^Ce^{-d/b}\\
S(d,b)& \ll b,\label{eq-sdb}
\end{align}
for $1\leq b\leq d$ and absolute constant $C>0$, with absolute
implied constants for the first, and an implied constant depending
only on $q$ for the second.
\par
Granting this, we have 
$$
d^2R(\sqrt{d},(\log d)^2)\ll \exp\Bigl(2\log d+2C\log\log d-
\frac{\sqrt{d}}{(\log  d)^2}\Bigr)\fleche{} 0,
$$
and all terms in $R_1$ are similarly trivially estimated, except for
$$
\exp((\log d)(1-e^{iu}))|\expect(e^{iu\varpi(\sigma_d)})|b^2d^{-1}
S(\sqrt{d},b)\ll
b^3d^{-1}\fleche{}0,
$$
using again the mod-Poisson convergence of $\varpi(\sigma_d)$.
\par
We now justify~(\ref{eq-sdb}) and~(\ref{eq-rankin}): for the former,
by~(\ref{eq-upper-mertens}), we have
$$
|S(\sqrt{d},b)|\leq \prod_{\deg(\irred)\leq b}{
\Bigl(1+\frac{1}{\nirred-1}\Bigr)
}\ll b,
$$
and for the latter, we need only a simple application of the
well-known Rankin trick: for any $\sigma\geq 0$, $d\geq 1$ and $g\in
\Fp_q[X]$, we have
$$
\charfun_{\deg(g)>d}\leq q^{\sigma(\deg(g)-d)},
$$
and hence, by multiplicativity, we get
$$
R(d,b)\leq q^{-\sigma d}\sum_{\degp{g}\leq b}{q^{(\sigma-1)\deg(g)}}
=q^{-\sigma d}\prod_{\deg(\irred)\leq b}{(1-\nirred^{\sigma-1})^{-1}},
$$
which we estimate further for $\sigma$ using
\begin{align*}
\prod_{\deg(\irred)\leq b}{(1-\nirred^{\sigma-1})^{-1}}
&=\exp\Bigl(\sum_{\deg(\irred)\leq b}{\sum_{k\geq 1}
{\frac{\nirred^{k(\sigma-1)}}{k}}}\Bigr)
\\
&\leq 
\exp\Bigl(C\sum_{j=1}^b{\frac{q^{j\sigma}}{j}}\Bigr)\leq
\exp(C'q^{\sigma b}\log b)
\end{align*}
for some absolute constants $C, C'>0$. Taking $\sigma=1/(b\log q)$
leads immediately to~(\ref{eq-rankin}).
\par
Finally, here is the computation of the characteristic function of the
cycle count of permutations without small parts that we used in the
proof of Proposition~\ref{prop-td}.

\begin{prop}\label{pr-permut}
For all   $d\geq 2$ and $b\geq 0$ such that $b\leq d$, we have
\begin{multline}
  \expect(e^{iu\varpi(\sigma_d)}\charfun_{\ellm{\sigma_d}>b})
  =\exp\Bigl(-e^{iu}\sum_{j=1}^b{\frac{1}{j}}\Bigr)
  \expect(e^{iu \varpi(\sigma_d)})\\
  +O(|\expect(e^{iu \varpi(\sigma_d)})|b^{1+\eps}d^{-1}+b^3(\log
  d)^{1/2} d^{-2}),\label{eq-goal}
\end{multline}
for any $\eps>0$, where the implied constant depend only on $\eps$.
\end{prop}

\begin{proof}
  This is essentially a sieve (or inclusion-exclusion) argument, which
  may well be already known (although we didn't find it explicitly in
  our survey of the literature). To simplify the notation, we will
  prove the statement by induction on $b$, although this may not be
  necessary; taking care of the error terms is then slightly more
  complicated, and readers should probably first disregard them to see
  the main flow of the argument.
\par
We denote
$$
\Phi_{d,b}(u)=\expect(e^{iu\varpi(\sigma_d)}\charfun_{\ellm{\sigma_d}>b}),
\quad
\Phi_d(u)=\Phi_{d,0}(u),\quad
\harm{b}=\sum_{j=1}^b{\frac{1}{j}}.
$$
\par
We will write
\begin{equation}\label{eq-induction}
\Phi_{d,b}
=\exp(-e^{iu}\harm{b})\Phi_d+|\Phi_d|E_{d,b}+F_{d,b},
\end{equation}
where $E_{d,b}$, $F_{d,b}\geq 0$; such an expression holds for $b=0$,
with $E_{d,0}=F_{d,0}=0$, and we will proceed inductively to obtain an
expression for $\Phi_{d,b}$ from that of $\Phi_{d',b-1}$, $d'\leq d$,
from which we will derive estimates for $E_{d,b}$ and $F_{d,b}$ in
general. Note that we can assume that $d$ is large enough (i.e.,
larger than any fixed constant), since smaller values of $d$ (and $b$)
are automatically incorporated by making the right-most implied
constant large enough in~(\ref{eq-goal}). Also, we can always write
such a formula with $|F_{d,b}|\ll b$, for some absolute constant,
since the characteristic functions $\Phi_{d,b}$ are bounded by $1$ and
$\exp(-e^{iu}\harm{b})\ll b$.
\par
Now, with these preliminaries settled, let $I$ be the set of
$b$-cycles in $\sym_d$; we write $\tau\mid \sigma$ (resp. $\tau\nmid
\sigma$) to indicate that $\tau\in I$ occurs (resp. does not occur) in
the decomposition of $\sigma$ in cycles. Then we have
\begin{align*}
\Phi_{d,b}(u)&=
\expect(e^{iu\varpi(\sigma_d)}\charfun_{\ellm{\sigma_d}>b})\\
&=\frac{1}{d!}\sum_{\stacksum{\ellm{\sigma}>b-1}
{\tau\in I\Rightarrow
    \tau\nmid \sigma}}{e^{iu\varpi(\sigma)}}
=\frac{1}{d!}\sum_{\ellm{\sigma}>b-1}
{e^{iu\varpi(\sigma)}\prod_{\tau\in I}{(1-\charfun_{\tau\mid \sigma})}}.
\end{align*}
\par
We expand the product as a sum over subsets $J\subset I$, and exchange
the two sums, getting
$$
\Phi_{d,b}(u)=
\frac{1}{d!}\sum_{J\subset I}{(-1)^{|J|}
\sum_{\stacksum{\ellm{\sigma}>b-1}
{\tau\in J\Rightarrow \tau\mid \sigma}}
{e^{iu\varpi(\sigma)}}}.
$$
\par
Now fix a $J\subset I$ such that the inner sum is \emph{not
  empty}. This implies of course that the support of the cycles in $J$
are disjoint, in particular that those cycles contribute $|J|$ to
$\varpi(\sigma)$. Moreover, if we call $A$ the complement of the union
of the support of the cycles in $J$, we have $|A|=d-|J|b$, and any
$\sigma$ in the inner sum maps $A$ to itself. Thus, by enumerating the
elements of $A$, we can map injectively those $\sigma$ to permutations
in $\sym_{d-|J|b}$, and the image of this map is exactly the set of
those $\sigma_1\in\sym_{d-|J|b}$ for which
$\ellm{\sigma_1}>b-1$. Moreover, if $\sigma$ maps to $\sigma_1$, we
have
$$
\varpi(\sigma)=|J|+\varpi(\sigma_1),
$$
and thus we get
$$
\sum_{\stacksum{\ellm{\sigma}>b-1}
{\tau\in J\Rightarrow \tau\mid \sigma}}
{e^{iu\varpi(\sigma)}}=
e^{iu|J|}
\sum_{\stacksum{\sigma\in \sym_{d-|J|b}}
{\ellm{\sigma}>b-1}}
{e^{iu\varpi(\sigma)}},
$$
and then
\begin{align*}
\Phi_{d,b}(u)&=
\sum_{J\subset I}{\frac{(d-|J|b)!}{d!}(-e^{iu})^{|J|}
\expect(e^{iu\varpi(\sigma_{d-|J|b})}
\charfun_{\ellm{\sigma_{d-|J|b}}>b-1})},\\
&=\sum_{J\subset I}{\frac{(d-|J|b)!}{d!}(-e^{iu})^{|J|}\Phi_{d-|J|b,b-1}(u)},
\end{align*}
the sum over $J$ being implicitly restricted to those subsets of $I$
for which there is at least one permutation in $\sym_d$ where all
cycles in $J$ occur.
\par
In particular, we have $|J|\leq d/b$ (so there is enough room to find
that many disjoint $b$-cycles), and if we denote by $S(k,b)$ the
number of possible such subsets of $I$ with $|J|=k$, we can write
$$
\Phi_{d,b}(u)
=\sum_{k=0}^{d/b}{S(k,b)\frac{(d-kb)!}{d!}(-e^{iu})^k
\Phi_{d-kb,b-1}(u)}
$$
\par
Now we claim that
$$
S(k,b)=\binom{d}{d-kb}\times \frac{(kb)!}{b^kk!}
=\frac{d!}{(d-kb)!b^kk!}.
$$
\par
Indeed, to construct the subsets $J$ with $|J|=k$, we can first select
arbitrarily a subset $A$ of size $d-kb$ in $\{1,\ldots, d\}$, and then
select, independently, an arbitrary set of $k$ disjoint $b$-cycles
supported outside $A$. The choice of $A$ corresponds to the binomial
factor above, and the second factor is clearly equal to the number of
permutations $\sigma\in \sym_{kb}$ which are a product of $k$ disjoint
$b$-cycles. Those are all conjugate in $\sym_{kb}$, and their
cardinality is given by~(\ref{eq-card-conj}), applied with $d$
replaced by $kb$ and all $r_j=0$ except for $r_b=k$.
\par
Consequently, we obtain the basic induction relation
$$
\Phi_{d,b}(u)
=\sum_{k=0}^{d/b}{\Bigl(\frac{-e^{iu}}{b}\Bigr)^k\frac{1}{k!}
\Phi_{d-kb,b-1}(u)}.
$$
\par
Before applying the induction assumption~(\ref{eq-induction}), we
shorten the sum over $k$ so that $\Phi_{d-kb,b-1}$ will remain close
to $\Phi_{d,b-1}$. For this, we use the inequality
$$
\Bigl|\sum_{k=0}^m{\frac{z^k}{k!}}-e^z\Bigr|\leq \frac{1}{m!},
$$
for $|z|\leq 1$, $m\geq 0$, as well as $|\Phi_{d-kb}(u)|\leq 1$, and
deduce that
\begin{equation}\label{eq-cont}
\Phi_{d,b}(u)
=\sum_{k=0}^{m}{\Bigl(\frac{-e^{iu}}{b}\Bigr)^k\frac{1}{k!}
\Phi_{d-kb,b-1}(u)}
+O\Bigl(\frac{1}{m!}\Bigr),
\end{equation}
for some $m$ to be specified later, subject for the moment only to the
condition $m<d/2b$, and an implied constant which is at most $1$.
\par
By~(\ref{eq-induction}), we have
$$
\Phi_{d-kb,b-1}(u)=\exp(-e^{iu}\harm{b-1})\Phi_{d-kb}(u)
+|\Phi_{d-kb}(u)|E_{d-kb,b-1}+F_{d-kb,b-1}.
$$
\par
Moreover, by~(\ref{eq-unif-char}), we also know that for $k\leq m$, we
have
\begin{equation}\label{eq-shift-char}
  \Phi_{d-kb}(u)=
  \expect(e^{iu\varpi(\sigma_{d})})\erreurm{\frac{kb}{d}}=
  \Phi_d(u)\erreurm{\frac{kb}{d}},
\end{equation}
with an absolute implied constant. Hence, we obtain
$$
\Phi_{d,b}(u)=\exp(-e^{iu}\harm{b-1})\Phi_d(u)M+R+S
$$
where
\begin{align*}
M&=\sum_{k=0}^{m}{\Bigl(\frac{-e^{iu}}{b}\Bigr)^k\frac{1}{k!}
\erreurm{\frac{bk}{d}}}\\
|R|&\leq \sum_{k=0}^m{\frac{1}{b^kk!}E_{d-kb,b-1}|\Phi_{d-kb}(u)|}\\
&=|\Phi_d(u)|\sum_{k=0}^m{\frac{1}{b^kk!}E_{d-kb,b-1}\Bigl(1+O\Bigl(
\frac{kb}{d}
\Bigr)\Bigr)}
\\
|S|&\leq \sum_{k=0}^m{\frac{1}{b^kk!}F_{d-kb,b-1}}+
\frac{1}{m!}.
\end{align*}
\par
We next write
$$
M=\exp\Bigl(-\frac{e^{iu}}{b}\Bigr)+
O\Bigl(\frac{1}{d}\sum_{k=1}^{m}{\frac{1}{b^{k-1}(k-1)!}}\Bigr)
+O\Bigl(\frac{1}{m!}\Bigr),
$$
where the implied constants are absolute, and deduce that
$$
\Phi_{d,b}(u)=\exp(-e^{iu}\harm{b})\Phi_d(u)+|\Phi_d(u)|M_1+R+S,
$$
with
$$
|M_1|\ll \frac{1}{m!}+d^{-1}e^{1/b},
$$
where the implied constant is absolute. The desired shape of the main
term is now visible, and it remains to verify that (for a suitable
$m$) the other terms are bounded as stated in the proposition.
\par
First, comparing with~(\ref{eq-induction}), with the terms in $R_1$
and $R$ contributing to $E_{d,b}$, while those in $S$ contribute to
$F_{d,b}$, we see that we have
$$
F_{d,b}\leq \sum_{k=0}^m{\frac{1}{b^kk!}F_{d-kb,b-1}}+
\frac{1}{m!}.
$$
\par
We now select $m=\lfloor \log d\rfloor$. Then, together with
$F_{d,0}=0$, we claim that this inductive inequality implies
\begin{equation}\label{eq-borne-f}
F_{d,b}\leq Cb^{3}(\log d)^{1/2}d^{-2},
\end{equation}
for $b\leq d$ and some absolute implied constant $C\geq 1$. For a
large enough value of $C$, note that this is already true for all
$d\leq d_0$, where $d_0$ can be any fixed integer. We select $d_0$ so
that
$$
\frac{1}{m!}\leq \frac{1}{d^2}, 
$$
for $d\geq d_0$, and we can thus assume that $d>d_0$ from now on.
\par
The desired bound holds, of course, for $b=0$. It is also trivial if
$b(\log d)\geq d/24$ (say), because we have observed at the beginning
that (\ref{eq-induction}) can be obtained with $F_{d,b}\ll b$. If it
is assumed to be true for all $d$ and $b-1$, we have for $b(\log
d)<d/24$ that
\begin{align*}
  F_{d,b}&\leq\frac{1}{m!}+ C\sum_{k=0}^m{\frac{1}{b^kk!}F_{d-kb,b-1}}
\\
  &\leq \frac{1}{m!}+\frac{C(\log d)^{1/2}(b-1)^{3}}{d^2}\sum_{k=0}^m{
    \frac{1}{b^kk!}\Bigl(1-\frac{kb}{d}\Bigr)^{-2}}.
\end{align*}
\par
We note the following simple inequalities
\begin{gather*}
(1-x)^{-1}\leq e^{2x},\quad \exp(x)\leq 1+\frac{3x}{2},\quad\text{ for
}0\leq x\leq 1/2,
\\
(x-1)e^{1/x}\leq x,\quad\quad\text{ for } 0\leq x\leq 1,
\end{gather*}
and from them we deduce that if $b(\log d)<d/24$ (so that $kb/d\leq
1/2$ for the values involved), we have
$$
\sum_{k=0}^m{\frac{1}{b^kk!}\Bigl(1-\frac{kb}{d}\Bigr)^{-2}} \leq
\sum_{k=0}^m{ \frac{1}{k!}\Bigl( \frac{\exp(4b/d)}{b} \Bigr)^k}
\leq \exp\Bigl(\frac{1}{b}+\frac{6}{d}\Bigr)
$$
and, hence (from the same simple inequalities) we get
\begin{multline*}
(b-1)^3\sum_{k=0}^m{\frac{1}{b^kk!}\Bigl(1-\frac{kb}{d}\Bigr)^{-2}}
\leq
(b-1)^{3/2}\times (b-1)\exp\Bigl(\frac{1}{b}\Bigr)\\
\times \Bigl((b-1)\exp\Bigl(\frac{1}{3d}\Bigr)\Bigr)^{1/2}
 \leq (b-1)^{3/2}b^{3/2}.
\end{multline*}
\par
By the choice of $d_0$, we deduce for $b\geq 1$ and $d>d_0$ that we
have
$$
F_{d,b}\leq d^{-2}(\log d)^{1/2}(1+Cb^{3/2}(b-1)^{3/2})
\leq Cd^{-2}b^3,
$$
(assuming again $C$ large enough), completing the verification
of~(\ref{eq-borne-f}) by induction.
\par
Finally, from~(\ref{eq-induction}) and the foregoing, we deduce
similarly that
$$
E_{d,b}\leq D\Bigl(\frac{1}{m!}+d^{-1}e^{1/b}\Bigr)+
\sum_{k=0}^m{\frac{1}{b^kk!}E_{d-kb,b-1}\Bigl(1+O\Bigl(
\frac{kb}{d}
\Bigr)\Bigr)},
$$
for some absolute constant $D\geq 0$. Fix $\eps>0$, and consider the
bound
$$
E_{d,b}\leq Cb^{1+\eps}d^{-1};
$$
then if $C\geq 1$, assuming it for $b-1$, we obtain the inductive
bound
\begin{align*}
E_{d,b}&\leq d^{-1}\Bigl\{
D(1+e^{1/b})+C(b-1)^{1+\eps}\sum_{k=0}^m{\frac{1}{b^kk!}\Bigl(1+\frac{\beta
  kb}{d}\Bigr)\Bigl(1-\frac{kb}{d}\Bigr)^{-1}}
\Bigr\}\\
&\leq d^{-1}\Bigl\{D(1+e^{1/b})+C(b-1)^{1+\eps}
\exp\Bigl(\frac{1}{b}+\frac{3(\beta+2)}{2d}\Bigr)\Bigr\}
\end{align*}
(using again the elementary inequalities above). Then for $d\geq
d_1(\eps)$, provided $C\geq 1$, we obtain
$$
E_{d,b}\leq Cb^{1+\eps}d^{-1},
$$
confirming the validity of this estimate.
\end{proof}

\begin{rem}
  Proposition~\ref{pr-permut} can itself be seen as an instance of
  mod-Poisson convergence, for the cycle count of randomly, uniformly,
  chosen permutations in $\mathfrak{S}_d$ without small
  cycles. 
\par
Precisely, let $\mathfrak{S}_d^{(b)}$ denote the set of $\sigma\in
\mathfrak{S}_d$ with $\ellm{\sigma}>b$. We then find first (by putting
$u=0$ in Proposition~\ref{pr-permut}) that
$$
\frac{|\mathfrak{S}_d^{(b)}|}{|\mathfrak{S}_d|}
\sim_{d,b\ra +\infty}
\exp\Bigl(-\sum_{j=1}^b{\frac{1}{j}}\Bigr)
\sim \frac{e^{-\gamma}}{b},
$$
provided $b$ is restricted by $b\ll d^{1/2-\eps}$ with $\eps>0$
arbitrarily small. Then, for arbitrary $u\in\Rr$ and $b$ similarly
restricted, we find that
$$
\frac{1}{|\mathfrak{S}_d^{(b)}|}
\sum_{\sigma \in \mathfrak{S}_d^{(b)}}{
  e^{iu\varpi(\sigma)}
}\sim_{d,b\ra +\infty}
\exp\Bigl((1-e^{iu})\sum_{j=1}^b{\frac{1}{j}}\Bigr)
\expect(e^{iu\varpi(\sigma_d)}),
$$
locally uniformly. Thus the mod-Poisson
convergence~(\ref{eq-mod-poisson-permut}) for $\varpi(\sigma_d)$
implies mod-Poisson convergence for the cycle count restricted to
$\mathfrak{S}_d^{(b)}$ as long as $b\ll d^{1/2-\eps}$, with limiting
function $1/\Gamma(e^{iu})$ and parameters
$$
\log d-\sum_{j=1}^b{\frac{1}{j}}\sim \log\frac{d}{b}.
$$
\par
It may be that the restriction of $b$ with respect to $d$ could be
relaxed. However, in the opposite direction, note that for $b=d-1$,
the number of $d$-cycles in $\mathfrak{S}_d$, i.e.,
$|\mathfrak{S}_d^{(d-1)}|$, is $(d-1)!$, so the ratio is $1/d$ which
is obviously not asymptotic with $e^{-\gamma}/(d-1)$.
\end{rem}

\begin{rem}\label{rm-mertens}
We come back to the asymptotic formula~(\ref{eq-mertens}), to explain
how it follows from Theorem~\ref{th-main} in the sharper form
\begin{equation}\label{eq-mertens2}
\prod_{\deg(\irred)\leq d}{\Bigl(1-\frac{1}{\nirred}\Bigr)}=
\exp\Bigl(-\sum_{1\leq j\leq d}{\frac{1}{j}}\Bigr)
\erreurm{\frac{1}{q^{d/2}}}.
\end{equation}
\par
Namely, it is very easy to derive this asymptotic up to some
constant:
$$
\prod_{\deg(\irred)\leq d}{\Bigl(1-\frac{1}{\nirred}\Bigr)}=
\exp\Bigl(\gamma_q-\sum_{1\leq j\leq d}{\frac{1}{j}}\Bigr)
\erreurm{\frac{1}{q^{d/2}}},
$$
where $\gamma_q$ is given by the awkward, yet absolutely convergent,
expression
\begin{equation}\label{eq-gamma-const}
\gamma_q=
\sum_{\irred}{\Bigl(\log\Bigl(1-\frac{1}{\nirred}\Bigr)
+\frac{1}{\nirred}\Bigr)}+
\sum_{j\geq 1}{\Bigl(\frac{\Pi_q(j)}{q^j}-\frac{1}{j}\Bigr)}.
\end{equation}
\par
From this, the flow of the proof leads to the mod-Poisson
limit~(\ref{eq-main}), with an additional factor
$\exp(-\gamma_qe^{iu})$ in the limit. But for $u=0$, both sides
of~(\ref{eq-main}) are equal to $1$, so we must have
$\exp(\gamma_q)=1$ for all $q$. (This is another interesting example
of the information coming from mod-Poisson convergence, which is
invisible at the level of the normal limit; note in particular that
this is really a manifestation of the random permutations.)
\end{rem}

\section{Final comments and questions}

Many natural questions arise out of this paper. The most obvious
concern the general notion of mod-Poisson convergence, and its
probabilistic significance and relation with other types of
convergence and measures of approximation (and similarly for
mod-Gaussian behavior). Already from~\cite{hwang}, it is clear that
mod-Poisson convergence should be a very general fact in the setting
of ``logarithmic combinatorial structures'', as discussed
in~\cite{abt}. 
\par
In the direction suggested by the Erd\H{o}s-Kac Theorem, there is a
very abundant literature concerning generalizations to additive
functions and beyond (see, e.g., the discussion at the end
of~\cite{granville-sound}), and again it would be interesting to know
which of those Central Limit Theorems extend to mod-Poisson
convergence, and maybe even more so, to know which \emph{don't}.
\par
In the direction of pursuing the analogy with distribution of
$L$-functions, the first thing to do might be to construct a proof of
the mod-Poisson Erd\H{o}s-Kac Theorem for integers which parallels
the one of the previous section. This does not seem out of the
question, but our current attempts suffer from the fact that the
associations of permutations in ``$\sym_{\log N}$'' to integers $n\leq
N$ that we have considered are ad-hoc (though potentially useful), and
do not carry the flavor of a generalization of the Frobenius. It is
then difficult to envision a further natural analogue of a unitary
matrix associated, say, with $\zeta(1/2+it)$. One can suggest a ``made
up'' matrix $U_t$ obtained by taking the zeros of $\zeta(s)$ close to
$t$, and wrapping them around the unit circle after proper rescaling,
but this also lacks a good a priori definition -- though this was
studied by Coram and Diaconis~\cite{coram-diaconis}, who obtained
extremely good numerical agreement; this is also close to the
``hybrid'' model for the Riemann zeta function of Gonek, Hughes and
Keating~\cite{ghk}.
\par
One may hope for more success in the case of finite fields in trying
to understand (for instance) families of $L$-functions of algebraic
curves in the limit of large genus, since the definition of a random
matrix from Frobenius does not cause problem there (though recall it
is really a \emph{conjugacy class}). However, although we have
Deligne's Equidistribution Theorem in the ``vertical'' direction $q\ra
+\infty$, and its proof is highly effective, it is not clear what a
suitable analogue of the quantitative ``diagonal''
equidistribution~(\ref{eq-quant-equid}) in Lemma~\ref{lm-distrib}
should be. More precisely, what condition should replace the
restriction to polynomials without small irreducible factors?  We do
not have clear answers at the moment, but we hope to make progress in
later work.
\par
Finally, it should be clear that analogues of mod-Gaussian and
mod-Poisson convergence exist, involving other families of probability
distributions. Some cases related to discrete variables are discussed
in~\cite[\S 5]{bkn}, and one may also define ``mod-stable''
convergence in an obvious way (though we do not have interesting
examples of these to suggest at the moment). It may be interesting to
investigate links between these various definitions; the last part of
Proposition~\ref{prop-cor-mod-poisson} suggests that there should
exist interesting relations.


\begin{thebibliography}{999}

\bibitem{abt}
R. Arratia, A.D. Barbour and S. Tavar\'e:
\textit{Logarithmic combinatorial structures: a probabilistic
  approach}, E.M.S. Monographs, 2003.

\bibitem{barbour-hall} A.D. Barbour and P. Hall: \textit{On the rate
    of Poisson convergence}, Math. Proc. Camb. Phil. Soc. 95 (1984),
  473--480.

\bibitem{bkn} A.D. Barbour, E. Kowalski and A. Nikeghbali:
  \textit{Mod-discrete expansions}, preprint (2009),
  \url{arXiv:0912.1886}

\bibitem{breiman} L. Breiman: \textit{Probability}, Classics
  in Applied Mathematics 7, SIAM, 1992.

\bibitem{coram-diaconis} M. Coram and P. Diaconis: \textit{New tests
    of the correspondence between unitary eigenvalues and the zeros of
    Riemann's zeta function}, J. Phys. A. 36 (2000), 2883--2906.

\bibitem{flajolet-soria} P. Flajolet and M. Soria: \textit{Gaussian
    limiting distributions for the number of components in
    combinatorial structures}, J. Combin. Theory Ser. A 53 (1990),
  165--182.

\bibitem{ghk} S. Gonek, C. Hughes and J. Keating: \textit{A hybrid
    Euler-Hadamard product for the Riemann zeta function}, Duke
  Math. J.  136 (2007), 507--549.

\bibitem{granville} A. Granville: \textit{The anatomy of integers and
    permutations}, preprint (2008),
  \url{http://www.dms.umontreal.ca/~andrew/PDF/Anatomy.pdf}

\bibitem{granville-sound} A. Granville and K. Soundararajan:
  \textit{Sieving and the Erd\H{o}s-Kac Theorem}, in
  ``Equidistribution in Number Theory, An Introduction'', edited by
  A. Granville and Z. Rudnick, Springer Verlag 2007.


\bibitem{hardy-wright} G.H. Hardy and E.M. Wright: \textit{An
    introduction to the theory of numbers}, 5th Edition, Oxford
  Univ. Press, 1979.

\bibitem{hwang} Hwang, H-K.: \textit{Asymptotics of Poisson
    approximation to random discrete distributions: an analytic
    approach}, Adv. Appl. Prob. 31 (1999), 448--491.

\bibitem{jkn} J. Jacod, E. Kowalski and A. Nikeghbali:
  \textit{Mod-Gaussian convergence: new limit theorems in probability
    and number theory}, to appear in Forum Math; \url{arXiv:0807.4739}

\bibitem{katzsarnak} N.M. Katz and P. Sarnak: \textit{Random matrices,
    Frobenius eigenvalues, and monodromy}, A.M.S Colloquium Publ. 45,
  A.M.S, 1999.

\bibitem{keating-snaith} J.P. Keating and N.C. Snaith: \textit{Random
    matrix theory and $\zeta(1/2+i t)$}, Commun. Math. Phys.
  \textbf{214}, (2000), 57-89.

\bibitem{Petrov95} V.V. Petrov: \textit{Limit Theorems of Probability
    Theory}, Oxford University Press, Oxford, 1995.

\bibitem{renyi-turan}
A. R\'enyi and P. Tur\'an: \textit{On a theorem of
  Erd\H{o}s-Kac}, Acta Arith. 4 (1958), 71--84.

\bibitem{rhoades}
R.C. Rhoades: \textit{Statistics of prime divisors in function
  fields}, Internat. J. Number Theory 5 (2009), 141--152.

\bibitem{roos} B. Roos: \textit{Sharp constants in the Poisson
    approximation}, Stat. Prob. Letters 52 (2001), 155--168.

\bibitem{rosen}
M. Rosen: \textit{A generalization of Mertens' theorem}, J. Ramanujan
Math. Soc. 14 (1999), 1--19.

\bibitem{tenenbaum} G. Tenenbaum: \textit{Introduction to
    analytic and probabilistic number theory}, Cambridge Studies
  Adv. Math. \textbf{46}, Cambridge Univ. Press, 1995.

\bibitem{ww} E.T. Whittaker and G.N. Watson: \textit{A course
    in modern analysis}, 4th Edition, Cambridge Math. Library,
  Cambridge Univ. Press, 1996.

\end{thebibliography}
\end{document}